\documentclass[12 pt]{amsart}
\usepackage{hyperref}

\usepackage{amsrefs}
\usepackage{thm-restate}

\usepackage{tikz}
\usetikzlibrary{trees}

\tikzstyle{level 1}=[level distance=3cm, sibling distance=2cm]
\tikzstyle{level 2}=[level distance=3cm, sibling distance=2cm]
\tikzstyle{level 3}=[level distance=3cm, sibling distance=2cm]
\tikzstyle{level 4}=[level distance=3cm, sibling distance=2cm]
\tikzstyle{level 5}=[level distance=3cm, sibling distance=2cm]

\tikzstyle{ns} = [text width=5em]
\tikzstyle{end} = [circle, minimum width=3pt,fill, inner sep=0pt]

\usepackage[colorinlistoftodos]{todonotes}

\newcommand{\norm}[1]{\left\lVert#1\right\rVert}

\usepackage{amsxtra}
\usepackage{geometry}
\usepackage{amssymb}
\usepackage{cite}
\usepackage{graphicx}
\usepackage[mathscr]{eucal}
\usepackage[normalem]{ulem}
\usepackage{tensor}
\RequirePackage{stix}

\usepackage{thmtools} 
\usepackage{thm-restate}

\usepackage{scalerel}

\theoremstyle{plain}
\newtheorem{Thm}{Theorem}

\theoremstyle{definition}
\newtheorem{Remark}[Thm]{Remark}

\renewcommand{\bar}{\overline}

\DeclareMathOperator{\rank}{rank}
\numberwithin{equation}{section}

\author[Brooks]{Jennifer Brooks}
\address{Brooks: Department of Mathematics, Brigham Young University,
Provo, UT, 84602 USA}
\email{jbrooks@mathematics.byu.edu}

\author[Curry]{Sean Curry}
\address{Curry: Department of Mathematics, Oklahoma State University, Stillwater, OK, 74078 USA}
\email{sean.curry@okstate.edu}

\author[Grundmeier]{Dusty Grundmeier}
\address{Grundmeier: Mathematics Department, Harvard University, Cambridge, MA, 02138 USA}
\email{deg@math.harvard.edu}

\author[Gupta]{Purvi Gupta}
\address{Gupta: Department of Mathematics, Indian Institute of Science, Bangalore, 560012}
\email{purvigupta@iisc.ac.in}

\author[Kunz]{Valentin Kunz}
\address{Kunz: Department of Mathematics, University of Manchester,
Manchester, UK}
\email{valentin.kunz@manchester.ac.uk}

\author[Malcom]{Alekzander Malcom}
\address{Malcom: Courant Institute of Mathematical Sciences,
New York University,
New York, NY 10012 USA}
\email{am12205@nyu.edu}

\author[Palencia]{Kevin Palencia}
\address{Palencia: Department of Mathematical Sciences,
Northern Illinois University,
De Kalb, IL 60115 USA}
\email{palencia@niu.edu}


\begin{document}
\title[Invariant Mappings]{Constructing Group-Invariant CR Mappings}

\thanks{This material is based upon work supported by the National Science Foundation under Grant Number DMS 1641020.}

\date{\today}

\begin{abstract}  
We construct CR mappings between spheres that are invariant under actions of finite unitary groups. In particular, we combine a tensoring procedure with D'Angelo's construction of a canonical group-invariant CR mapping to obtain new invariant mappings. We also explore possible gap phenomena in this setting. 
\end{abstract}

\maketitle

\section{Introduction} \label{s:intro}

Let $S^{2n-1}$ be the unit sphere in $\mathbb{C}^n$. When $N,n \geq 2$, a longstanding problem in several complex variables is to classify CR mappings from $S^{2n-1}$ to $S^{2N-1}$. Much is known about this problem (see \cites{D:book1,D:book2, F:survey} and the references therein). In particular, the solution depends on the assumed regularity. Forstneri{\v c} proved that a sufficiently smooth CR mapping between spheres must be rational. Thus, in this paper, we restrict our attention to rational mappings (see D'Angelo \cites{D:book1} for an extensive discussion of the rational case). When $N<n$, there is only the constant map. Alexander \cites{A:auto} proved in the equidimensional setting ($n=N$) that all non-constant maps are automorphisms and hence are spherically equivalent to the linear embedding. In fact, for $3\leq n \leq N < 2n-1$, all smooth CR maps are spherically equivalent to $z \mapsto  (z,\bf{0})$ (see \cites{H:lemma,HJ:gaps,W,Fa}).  These intervals of values for $N$ where no new maps appear are referred to as ``gaps". When $N=2n-1$ this first gap terminates. For example, when $n=3$ and $N=2n-1=5$, we have a {\it Whitney map} $$(z_1,z_2,z_3)\mapsto (z_1,z_2,z_1 z_3, z_2 z_3, z_3^2)$$ along with the linear embedding. There has been extensive progress studying rigidity and gap phenomena for sphere maps in recent years, and we refer the reader to \cites{BG, DL:complexity,E:sos,EHZ:rigid,GH:sos, H:lemma,HJ:gaps, HJY} and the references therein. As the codimension $N-n$ gets larger, there are many inequivalent maps. Thus it is natural to impose additional conditions. 

In this paper, we impose a natural symmetry condition (see \cites{DX1, DX2} for more on symmetries of CR mappings of spheres). Suppose $\Gamma$ is a finite subgroup of $U(n)$. We study non-constant, smooth, CR mappings $f: S^{2n-1} \to S^{2N-1}$ such that $f\circ \gamma = f$ for all $\gamma \in \Gamma$. In this setting, D'Angelo and Lichtblau \cites{DL:spaceforms, D:book2} proved that the only finite subgroups of $U(n)$ that admit non-constant, smooth, CR mappings to some sphere are equivalent to one of the following:
\begin{itemize}
    \item $\langle \omega I_n \rangle$ where $\omega$ is a primitive $p$-th  root of unity,
    \item $\langle \omega I_k \oplus \omega^2 I_{n-k} \rangle$ where $\omega$ is a primitive $p$-th root of unity for $p$ odd,
    \item $\langle \omega I_j \oplus \omega^2 I_{k} \oplus \omega^4 I_{n-j-k}\rangle$ where $\omega$ is a primitive 7-th root of unity;
\end{itemize}
we refer to such subgroups as {\it admissible subgroups.}
For any fixed-point-free, finite subgroup $\Gamma$, D'Angelo gives a construction of a canonical group-invariant CR mapping from a sphere to a hyperquadric. In general, it is difficult to determine the exact target hyperquadric (see \cites{G:thesis,G:sigpairs,GLW} for more in this direction). However, in the case of an admissible subgroup, the target is a sphere. Let $N(\Gamma)$ denote the dimension of the target of D'Angelo's map (see Section \ref{s:background} for a detailed construction). In the present paper, we fix an admissible subgroup $\Gamma$ and construct non-constant CR mappings $f: S^{2n-1}/\Gamma \to S^{2N-1}$. For each $\Gamma$, we then study possible minimal embedding dimensions $N$. In particular, we show that beyond some point there are no more gaps.

\begin{restatable}{Thm}{gaptermination}
\label{T:bound}
If $\Gamma$ is an admissible subgroup and $N \geq N(\Gamma)^2-2N(\Gamma)+2$, then there exists a smooth CR mapping $f:S^{2n-1}/\Gamma \to S^{2N-1}$ for which $N$ is the minimal embedding dimension.
\end{restatable}

Our result and proof are inspired by D'Angelo and Lebl's proof \cites{DL:complexity} of gap termination in the non-invariant setting. In particular, we follow their clever approach using tensoring and  the solution of the postage stamp problem. The new contribution of this paper is to apply their tensoring procedure in the group-invariant setting. When $n=2$, we can improve our bound for $N$.

\begin{restatable}{Thm}{dimtwoproof}
\label{T:dim2}
 If $\Gamma$ is an admissible subgroup and $N \geq 2N({\Gamma})-1$, then there is a smooth CR mapping $f:S^{3}/\Gamma \to S^{2N-1}$ for which $N$ is the minimal embedding dimension.
\end{restatable}

We conclude the introduction by outlining the rest of the paper. In Section \ref{s:background}, we introduce D'Angelo's construction and the relevant background material. In Section \ref{s:construction}, we combine the canonical invariant mapping with a tensorial construction to generate new invariant mappings to spheres. In Section \ref{s:examples}, we present examples with $n=3$ that suggest a better bound is possible in Theorem \ref{T:bound}. Finally, in Section \ref{s:future}, we discuss future directions for this work. 

This project began as part of the Mathematics Research Community (MRC) program on ``New Problems in Several Complex Variables." We would like to thank the American Mathematical Society for creating this wonderful program and supporting our continued collaboration.

\section{Background and Setup} \label{s:background}

Let $Q(a,b)$ denote the hyperquadric with $a$ positive and $b$ negative terms in its defining equation; namely,
\begin{equation*}
    Q(a,b)=\left\{ z \in \mathbb{C}^{a+b}: \sum_{j=1}^a|z_j|^2-\sum_{j=a+1}^{a+b}|z_j|^2=1\right\}.
\end{equation*} Of course, $S^{2n-1}=Q(n,0)$.

Let $\Gamma$ be a fixed-point-free, finite subgroup of the unitary group $U(n)$. We define the following real-valued, $\Gamma$-invariant polynomial:
\begin{equation} \label{e:phi}
    \Phi_{\Gamma}(z, \bar{z}) = 1 - \prod_{\gamma \in \Gamma}{(1-\langle \gamma z, z\rangle)}.
\end{equation}
Expanding the product in \eqref{e:phi} gives $$\Phi_{\Gamma}(z,\bar{z})=\sum_{\alpha, \beta} {c_{\alpha\beta}z^{\alpha}\bar{z}^\beta}.$$ Because $\Phi_{\Gamma}$ is real-valued, the matrix of coefficients $(c_{\alpha \beta})$ is Hermitian. We define the {\it rank} of $\Phi_{\Gamma}$ to be the rank of this underlying coefficient matrix. Furthermore, we define $N^{+}(\Gamma)$ to be the number of positive eigenvalues of $(c_{\alpha \beta})$ and $N^{-}(\Gamma)$ to be the number of negative eigenvalues of $(c_{\alpha \beta})$. Define $N(\Gamma)= N^{+}(\Gamma)+N^{-}(\Gamma)$.

Diagonalizing the matrix of coefficients, we get
\begin{equation}
    \Phi_{\Gamma}(z,\bar{z}) =\| F(z) \|^2- \| G(z) \|^2
\end{equation}
where $F$ and $G$ are $\Gamma$-invariant, holomorphic polynomial mappings with linearly independent components. Therefore, from $\Phi_{\Gamma}$, we get an associated $\Gamma$-invariant CR mapping $$\phi_{\Gamma}=F \oplus G: S^{2n-1}/\Gamma \to Q(N^{+}, N^{-}).$$ We call this mapping the {\it canonical} group-invariant CR mapping associated with the subgroup $\Gamma$. When $N^{-}=0$, we get a non-constant CR mapping to the sphere $S^{2N(\Gamma)-1}$. 

We next reformulate our problem in terms of real polynomials. Suppose $\phi: \mathbb{C}^n \to \mathbb{C}^N$ and $\phi(S^{2n-1}) \subseteq S^{2N-1}$. Thus $\norm{\phi(z)}^2=1$ when $\norm{z}^2=1$. Furthermore, suppose $\phi$ is a monomial mapping (i.e. the components are monomials). Then $\phi(z)= (\ldots, c_{\alpha} z^{\alpha}, \ldots)$. Thus, $$\sum_{\alpha} |c_{\alpha}|^2|z^{\alpha}|^2=1 \text{ when } \sum_{j=1}^n|z_j|^2=1.$$ Setting $x_j=|z_j|^2$, we get $$\sum_{\alpha} |c_{\alpha}|^2x^{\alpha}=1 \text{ when } \sum_{j=1}^n x_j=1.$$ Thus we have a correspondence between monomial CR mappings $\phi$ of spheres and real polynomials $p$ with non-negative coefficients such that 
\begin{equation} \label{e:hyperplane}
    p(x_1, \ldots, x_n)=1 \text{ when } x_1+\dots+x_n=1.
\end{equation} If $\phi$ is additionally $\Gamma$-invariant, then $p$ is also $\Gamma$-invariant. In the monomial setting, we define $f_{\Gamma}(x_1, \ldots, x_n)$ to be the polynomial associated with the canonical mapping; namely,
\begin{equation*}
    f_{\Gamma}(x_1,\ldots, x_n)=f_{\Gamma}(|z_1|^2, \ldots, |z_n|^2)=\Phi_{\Gamma}(z,\bar{z}).
\end{equation*} Finally, observe that the rank of $\Phi_{\Gamma}$ is the number of independent monomials in $f_{\Gamma}$. 
We work with the real polynomials for the proofs of the main results.

We briefly pause to give examples. Suppose $n=2$. Then according to D'Angelo and Lichtblau, there are two groups we need to consider: $\langle \omega I_2 \rangle$ and $\langle \omega I_1 \oplus \omega^2 I_1 \rangle$. First, let $\Gamma(p,1)=\langle \omega I_2 \rangle$ where $\omega$ is a primitive $p$-th root of unity. The canonical mapping $\phi_{\Gamma(p,1)}: S^{3}/ \Gamma(p,1) \to S^{2(p+1)-1}$ is given by 
\begin{equation} \label{e:mapp1}
(z,w) \mapsto \left(z^p, \sqrt{\binom{p}{1}} z^{p-1}w, \cdots, \sqrt{\binom{p}{p-1}}z w^{p-1}, w^p\right).
\end{equation}
The corresponding real polynomial $f_{p,1}$ is given by
\begin{equation}
f_{p,1}(x,y):=f_{\Gamma(p,1)}(x,y)=\sum_{k=0}^p{\binom{p}{k}x^{p-k}y^k}=(x+y)^p.
\end{equation}

Next, let $\Gamma(p,2)= \langle \omega I_1 \oplus \omega^2 I_1 \rangle$ for $\omega$ a primitive $p$-th root of unity for $p$ odd. The canonical mapping $\phi_{\Gamma(p,2)}: S^{3} \to S^{p+2}$ is given by
$$(z,w) \mapsto \left(z^{p}, \sqrt{c_1} z^{p-2} w, \cdots, \sqrt{c_{\frac{p-1}{2}}} z w^{\frac{p-1}{2}}, w^p  \right)$$
where $$c_k=\frac{p}{p-k}\binom{p-k}{k}$$ for $1\leq k \leq \frac{p-1}{2}$.
As before, the corresponding real polynomial $f_{p,2}$ is given by 

\begin{equation}
f_{p,2}(x,y):= f_{\Gamma(p,2)}(x,y)=x^p + c_1 x^{p-2}y + \dots + c_{\frac{p-1}{2}}xy^{\frac{p-1}{2}}+y^p.
\end{equation}

When $n=3$, there are 4 inequivalent possibilities for $\Gamma$: $\langle \omega I_3 \rangle$ for $\omega$ any primitive $p$-th root of unity, $\langle \omega I_2 \oplus \omega^2 I_1 \rangle$ or $\langle \omega I_1 \oplus \omega^2 I_2 \rangle$ for $\omega$ a primitive $p$-th root of unity for $p$ odd, or $\langle \omega I_1 \oplus \omega^2 I_1 \oplus \omega^4 I_1 \rangle$ for $\omega$ a primitive 7-th root of unity. For the first case, we get $$f_{\langle \omega I_3 \rangle}(x_1,x_2,x_3)=(x_1+x_2+x_3)^p.$$ For cases 2 and 3, we get 
\begin{align*}
    f_{\langle \omega I_2 \oplus\omega^2 I_1 \rangle}(x_1,x_2,x_3)& =f_{p,2}(x_1+x_2,x_3)\\
    f_{\langle \omega I_1 \oplus\omega^2 I_2 \rangle}(x_1,x_2,x_3)& =f_{p,2}(x_1,x_2+x_3).
\end{align*}In the last case, we give the canonical mapping. Let $$\Gamma(7,2,4)=\langle \omega I_1 \oplus \omega^2 I_1 \oplus \omega^4 I_1 \rangle=\left\langle \begin{pmatrix} \omega & 0 & 0\\
0 & \omega^2 & 0\\
0 & 0 & \omega^4 \\
\end{pmatrix}\right\rangle$$ where $\omega$ is a primitive 7-th root of unity. The canonical mapping $\phi_{\Gamma(7,2,4)}: S^5 \to S^{33}$ associated to $\Gamma(7,2,4)$ is given by
\begin{align*}
    (z_1,z_2,z_3) \mapsto& (z_1^7, \sqrt{7} z_1^5 z_2, \sqrt{14} z_1^3 z_2^2, \sqrt{7} z_1 z_2^3, z_2^7, \sqrt{7} z_1^3 z_3,\\
    &\sqrt{14} z_1 z_2 z_3, \sqrt{7}z_1^2 z_2^4 z_3,\sqrt{7} z_2^5 z_3,\sqrt{7} z_1^4 z_2 z_3^2, \sqrt{7} z_1^2 z_2^2 z_3^2,\\
    &\sqrt{14} z_2^3 z_3^2,\sqrt{14} z_1^2 z_3^3, \sqrt{7} z_2 z_3^3,\sqrt{7} z_1 z_2^2 z_3^4, \sqrt{7} z_1 z_2^5, z_3^7),
\end{align*} 
with corresponding polynomial
\begin{align*}
f_{7,2,4}(x_1,x_2,x_3)&=x_1^7+ 7 x_1^5 x_2+ 14 x_1^3 x_2^2+ 7 x_1 x_2^3+ x_2^7+ 7 x_1^3 x_3+14 x_1 x_2 x_3\\
&+ 7 x_1^2 x_2^4 x_3+7 x_2^5 x_3+7 x_1^4 x_2 x_3^2+ 7 x_1^2 x_2^2 x_3^2+14 x_2^3 x_3^2+14 x_1^2 x_3^3\\&+7 x_2 x_3^3+7 x_1 x_2^2 x_3^4+ 7 x_1 x_2^5+ x_3^7.
\end{align*}

We conclude this section with an example of using tensoring to construct new maps from old maps. Consider $n=2$, $p=3$, and $q=2$. Then $$f_{3,2}(x,y)=x^3+3xy +y^3.$$ Now we multiply $y^3$ by $f_{3,2}(x,y)$ to get another $\Gamma(3,2)$-invariant polynomial with non-negative coefficients that is 1 on $x+y=1$. Thus, we get $$x^3+3xy +y^3(x^3+3xy +y^3)=x^3+3xy +x^3 y^3+3 xy^4+y^6,$$ and hence the corresponding CR mapping from $S^3$ to $S^9$ is given by $$(z,w)\mapsto (z^3, \sqrt{3} z w, z^3 w^3, \sqrt{3} z w^4, w^6).$$

\section{Proofs of the main results} \label{s:construction}

In this section, we prove Theorems~\ref{T:bound} and \ref{T:dim2}. In view of the correspondence discussed in Section~\ref{s:background}, we work entirely with real polynomials. Because the number of independent monomials in a real polynomial $g$ is the same as the number of independent components of the corresponding monomial mapping, this number is called the {\em rank} of $g$ and is denoted by $\rank(g)$. The proof of Theorem~\ref{T:bound} follows the D'Angelo and Lebl proof in the non-invariant setting from \cites{DL:complexity}.

\gaptermination*

\begin{proof} We show that for any $n\geq 2$, any finite subgroup $\Gamma$ of $U(n)$ in the D'Angelo--Lichtblau list, and any $N\geq N(\Gamma)^2-3N(\Gamma)+2$, there is a $\Gamma$-invariant, real polynomial with non-negative coefficients, satisfying \eqref{e:hyperplane}, and of rank $N$.

 Let $f_\Gamma$ denote the real polynomial corresponding to the squared-norm of the canonical CR mapping $\phi_\Gamma$. Then, there exists $p\in\mathbb N$ such that, in multi-index notation,
\begin{equation}
    f_\Gamma(x)=\sum_{|\alpha|\leq p} c_{\alpha} x^\alpha,
\end{equation}
where $c_{(p,0,...,0)}>0$. Note that $\rank(f_\Gamma)=N(\Gamma)$. 

Now, given a $\Gamma$-invariant, real polynomial $g$ of total degree $d$, satisfying \eqref{e:hyperplane} and containing a pure top-degree monomial in $x_1$, say $cx_1^d$, define 
\begin{eqnarray}
(V\!g)(x)&=&g(x)+\frac{c}{2}x_1^d\left(-1+
 f_\Gamma(x)\right),\\
(W\!g)(x)&=& g(x)+cx_1^d\left(-1+
 f_\Gamma(x)\right).
\end{eqnarray}
Then, $V\!g$ and $W\!g$ are $\Gamma$-invariant polynomials of degree $d+p$ that satisfy \eqref{e:hyperplane} and contain a pure top-degree monomial in $x_1$. Further, if $g$ has non-negative coefficients, so do $V\!g$ and $W\!g$, and
\begin{eqnarray*}
\rank(V\!g)&=&\rank(g)+N(\Gamma),\\
\rank(W\!g)&=&\rank(g)+N(\Gamma)-1.
\end{eqnarray*} 
Viewing $V$ and $W$ as operators, we use the notation $V^jW^kg$ to denote the $k$-times iterated application of $W$, followed by the $j$-times iterated application of $V$ to $g$. Thus,
\begin{equation*}
\rank(V^jW^kf_\Gamma)=jN(\Gamma)+k(N(\Gamma)-1)+N(\Gamma).
\end{equation*}

The result now follows by invoking Sylvester's solution to the postage stamp problem, which says that for any co-prime positive integers $A,B$, any integer $N\geq AB-A-B+1$ can be written as $jA+kB$ for some non-negative integers $j$ and $k$. In our case $A=N(\Gamma)$, $B=N(\Gamma)-1$, and, thus,  $AB-A-B+1=N(\Gamma)^2-3N(\Gamma)+2$. Hence, every rank beyond $N(\Gamma)^2-2N(\Gamma)+2$ is possible.
\end{proof}
\medskip

We now prove Theorem~\ref{T:dim2}.

\dimtwoproof*

\begin{proof} From \cites{DL:spaceforms}, we have that $\Gamma$ is equivalent to $\Gamma(p,1)$ or $\Gamma(p,2)$. From \cites{G:sigpairs}, we have that the minimal embedding dimension is preserved under a change of coordinates. Thus, we can assume without loss of generality that $\Gamma$ is either $\Gamma(p,1)$ or $\Gamma(p,2)$. 

For convenience, we switch our notation from $(x_1,x_2)$ to $(x,y)$ to denote the coordinates of a point in $\mathbb R^2$. First, suppose $\Gamma=\Gamma(p,1)$ for any fixed $p\geq 2$. In this case, $N(\Gamma)=p+1$, and the real polynomial corresponding to $\Phi_\Gamma$ is $$f_{p,1}(x,y)=(x+y)^p=\sum_{k=0}^{p} c_k x^{p-k}y^k,$$ 
where $c_k=\binom{p}{k}$. Let
\begin{equation}\label{eq:g11}
    g_1(x,y)=f_{p,1}(x,y)+x^p(-1+f_{p,1}(x,y))=\sum_{k=0}^{p}c_k x^{2p-k}y^k +\sum_{k=1}^{p}c_k x^{p-k}y^k.
\end{equation}
Then, $g_1$ is a $\Gamma$-invariant, real polynomial with non-negative coefficients, satisfying \eqref{e:hyperplane} and of rank $2p+1=2N(\Gamma)-1$.

Now, for each $j$ with $2\leq j\leq p$, we multiply the $(N(\Gamma)+j)$-th term of $g_1$ by $f_{p,1}$ to obtain
\begin{eqnarray*}
    g_j(x,y)&=& g_1(x,y)+c_jx^{p-j}y^j\left(-1+f_{p,1}(x,y)\right)\\
    &=&\sum_{k=0}^{p} c_k x^{2p-k}y^k+\sum_{k=0}^{p}c_j c_k x^{2p-j-k}y^{j+k} 
    +\sum_{\substack{k=1 \\ k \neq j}}^{p} c_k x^{p-k}y^k.
\end{eqnarray*}
Observe that the terms in the second sum on the last line above are monomials of the form $x^{2p-(j+k)}y^{j+k}$. When $j+k\leq p$, or, equivalently, $k\leq p-j$, such monomials also appear in the first sum on the last line above. Thus, collecting like terms, we see that for $j=2, \ldots ,p$, $g_j$ is a $\Gamma$-invariant, real polynomial with non-negative coefficients,  satisfying \eqref{e:hyperplane} and of rank $$(p+1)+j+(p-1)= 2N(\Gamma)-2+j.$$ We have obtained invariant polynomials of ranks $2N(\Gamma)-1,...,2N(\Gamma)-1+p-1=3N(\Gamma)-3$. 

For $N \geq 3N(\Gamma)-2$, we describe an iterative procedure for obtaining an invariant polynomial of rank $N$. For a fixed $N\geq 3N(\Gamma)-2$, let $j$ be the unique integer between $1$ and $N(\Gamma)-1$ such that $N -(2N(\Gamma)-2)\equiv j \pmod{N(\Gamma) - 1}$.  Thus
$$N=2N(\Gamma)-2+j+d(N(\Gamma)-1)$$
for some integer $d\geq 1$. We construct the desired polynomial iteratively. Let $g_j^{(0)}=g_j$. Let 
        $$
      g_j^{(1)}(x,y)=g_j^{(0)}(x,y)+x^{2p}(-1+f_{p,1}(x,y)).
    $$
Then, $g_j^{(1)}$ is an invariant polynomial, satisfying \eqref{e:hyperplane} and  consisting of exactly $2N(\Gamma)-2+j+N(\Gamma)-1$ independent monomials. This is because the polynomial $x^{2p}(f_{p,1}(x,y))$ consists of exactly $N(\Gamma)$ monomials, all of degree at least $2p$ in $x$, and all the terms of $g_j^{(0)}(x,y)-x^{2p}$ are of degree strictly less than $2p$ in $x$. Since $g_j^{(1)}$ contains the term $x^{3p}$, we can repeat this process. Repeating this process $d-1$ times gives an invariant polynomial satisfying \eqref{e:hyperplane}, and consisting of $N=2N(\Gamma)-2+j+d(N(\Gamma)-1)$ independent monomials. This completes the case of $\Gamma=\Gamma(p,1)$.

Next, suppose $\Gamma=\Gamma(p,2)$ for any fixed odd $p\geq 3$. The construction is quite similar to the previous case. In this case, the invariant real polynomial corresponding to $\Phi_\Gamma$ is
$$f_{p,2}(x,y)=\sum_{k=0}^{N(\Gamma) - 2} c_k x^{p-2k}y^k + y^p,$$
for some positive constants $c_k$. Once again, let
\begin{equation*}
g_1(x,y)=f_{p,2}(x,y)+x^p(-1+f_{p,2}(x,y))=\sum_{k=0}^{N(\Gamma)-2}c_k x^{2p-2k}y^k + x^py^p+\sum_{k=1}^{N(\Gamma)-2}c_k x^{p-2k}y^k +y^p .
\end{equation*}
Then $g_1$ is a $\Gamma$-invariant, real polynomial with non-negative coefficients, satisfying \eqref{e:hyperplane} and of rank $2N(\Gamma)-1$. Now, for each $j$ with $1 \leq j \leq N(\Gamma)-2$, let
\begin{eqnarray}
    g_{j+1}(x,y)&=& g_1(x,y)+c_jx^{p-2j}y^j\left(-1+f_{p,2}(x,y)\right)\notag\\
    &=&\sum_{k=0}^{N(\Gamma)-2} c_k x^{2p-2k}y^k+x^py^p+\sum_{k=0}^{N(\Gamma)-2}c_j c_k x^{2p-2j-2k}y^{j+k} \label{e:sums}\\
    &&{}+ c_j x^{p-2j}y^{p+j}+\sum_{\substack{k=1 \\ k \neq j}}^{N(\Gamma)-2} c_k x^{p-2k}y^k +y^p.\notag
\end{eqnarray}
Once again, the second sum in \eqref{e:sums} involves monomials of the form $x^{2p-2(j+k)}y^{j+k}$ that coincide with the monomials occuring in the first sum whenever $j+k \leq N(\Gamma)-2$, or $k \leq N(\Gamma)-2-j$. Thus, for $j=1,..,N(\Gamma)-2$, $g_{j+1}$ is a $\Gamma$-invariant, real polynomial with non-negative coefficients, satisfying \eqref{e:hyperplane} and of rank
$$N(\Gamma)+\left(N(\Gamma)-1)-(N(\Gamma)-1-j\right)+1+N(\Gamma)-3+1=2N(\Gamma)-1+j.$$ We have produced invariant polynomials of ranks $2N(\Gamma)-1,...,3N(\Gamma)-3$. For $N\geq 3N(\Gamma)-2$, we apply the same procedure as in the case of $\Gamma(p,1)$.
\end{proof}

\begin{Remark}
The bound found in Theorem 2 can be written as $N \geq 2p+1$ for $\Gamma = \langle \omega I_2\rangle$ and $N\geq p+2$  for $\Gamma = \langle \omega I_1 \oplus \omega^2 I_1 \rangle$.
\end{Remark}

\section{Discussion and Examples} \label{s:examples}

Theorem 1 gives a rather simple argument to show that beyond a certain critical value, all natural numbers $N$ are possible minimal embedding dimensions for group-invariant sphere maps.  However, the bound on $N$ is likely far from optimal. Theorem 2 is a refinement for the two-dimensional case, showing that all $N \geq 2N(\Gamma)-1$ are minimal embedding dimensions. It is natural to ask whether this better result holds for $n>2$.  Although we do not answer this question in this paper, we give some examples for $n=3$ that indicate that many of the apparent gaps in possible minimal embedding dimension can be filled using other more sophisticated schemes. 

Consider first $\Gamma=\langle \omega I_2 \oplus \omega^2 I_1\rangle$ for $\omega$ a primitive cube root of unity. The canonical $\Gamma$-invariant polynomial $f_{3,2}(x_1+x_2,x_3)$ has rank $7$. Now, if $g$ is any $\Gamma$-invariant polynomial of rank $\rho$ with non-negative coefficients taking the value $1$ when $x_1+x_2+x_3=1$, then if we multiply the term with the highest degree in $x_1$ by $f_{3,2}(x_1+x_2,x_3)$, we obtain a new polynomial satisfying these same conditions and having rank $\rho + 6$. Thus once we obtain such a polynomial of a certain rank, we can always obtain a polynomial satisfying these same conditions and having any larger rank that is in the same congruence class modulo 6. We therefore need only describe how to construct one polynomial with rank in each of the $6$ congruence classes.   

As in the proofs above, we use two kinds of operations on a polynomial $g$. For the first, we take a term of $g$ and simply multiply it by $f_{3,2}(x_1+x_2,x_3)$.  For the second, we take a term of $g$ and ``split" it, writing it as a sum of two identical monomials.  We then multiply just one of these terms by $f_{3,2}(x_1+x_2,x_3)$. The number of terms introduced by each of these operations depends on the monomial at which we perform the operation because in general some of the terms generated will combine with terms already appearing in the polynomial.  

Using these operations, we construct a sequence of $\Gamma$-invariant polynomials taking the value $1$ when $x_1+x_2+x_3=1$ and having ranks in each of the congruence classes modulo $6$ as follows: Begin with $f_{3,2}(x_1+x_2,x_3)$.  Next, either multiply by $f_{3,2}(x_1+x_2,x_3)$ at $x_1^3$ or ``split" and multiply by $f_{3,2}(x_1+x_2,x_3)$ at $x_1^3$ to obtain polynomials with ranks 13 and 14, respectively.  We then take these polynomials and multiply by $f_{3,2}(x_1+x_2,x_3)$ at the monomial $x_1^2 x_2$ to obtain polynomials with ranks 15 and 16, and so on. Thus, together with the process described above, we obtain polynomials of all ranks greater than or equal to $13 = 2N(\Gamma)-1$. See Figure 1. 
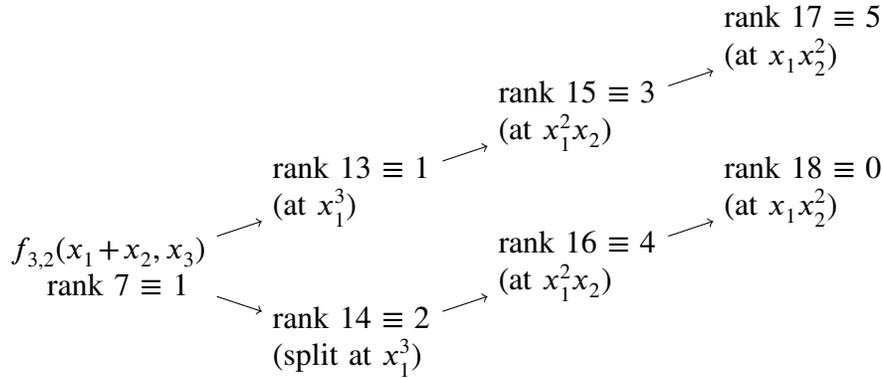
\begin{figure}[h]
\centering
\begin{tikzpicture}[grow=right]
\node[ns] {\hspace{-18pt} $f_{3,2}(x_1+x_2,x_3)$ rank~7~$\equiv$~1}
    child[->] {
        node[ns] {rank~14~$\equiv$~2 (split~at~$x_1^3$)}        
            child[edge from parent/.style={draw=none}] {
            }
            child {
                    node[ns] {rank~16~$\equiv$~4 (at~$x_1^2 x_2$)}        
                    child[edge from parent/.style={draw=none}] {
                    }
                    child {
                        node[ns] {rank~18~$\equiv$~0 (at~$x_1 x_2^2$)}
                    }
            }
    }
    child[->] {
        node[ns] {rank~13~$\equiv$~1 (at~$x_1^3$)}        
            child[edge from parent/.style={draw=none}] {
            }
            child {
                    node[ns] {rank~15~$\equiv$~3 (at~$x_1^2 x_2$)}        
                    child[edge from parent/.style={draw=none}] {
                    }
                    child {
                        node[ns] {rank~17~$\equiv$~5 (at~$x_1 x_2^2$)}
                    }
            }
    };
\end{tikzpicture}
\caption{Possible ranks for polynomials invariant under $\Gamma = \langle \omega I_2 \oplus \omega^2 I_1 \rangle$ for $\omega$ a primitive cube root of unity}
\end{figure}

We go through a similar process to obtain polynomials invariant under the action of $\Gamma =\langle \omega I_2 \oplus \omega^2 I_1 \rangle$ for $\omega$ a primitive fifth root of unity. We now begin with $f_{5,2}(x_1+x_2,x_3)$, which has rank 13.  We must therefore obtain $\Gamma$-invariant polynomials with ranks in each of the congruence classes modulo 12. Figure 2 shows the ranks achieved. Note that we do {\it not} obtain all ranks past $25=2N(\Gamma)-1$ through this process -- rank 27 is missing. However, we do obtain all ranks greater than or equal to $28$. This construction is an improvement over Theorem 1 which only guarantees that we can obtain ranks greater than or equal to $145$. 

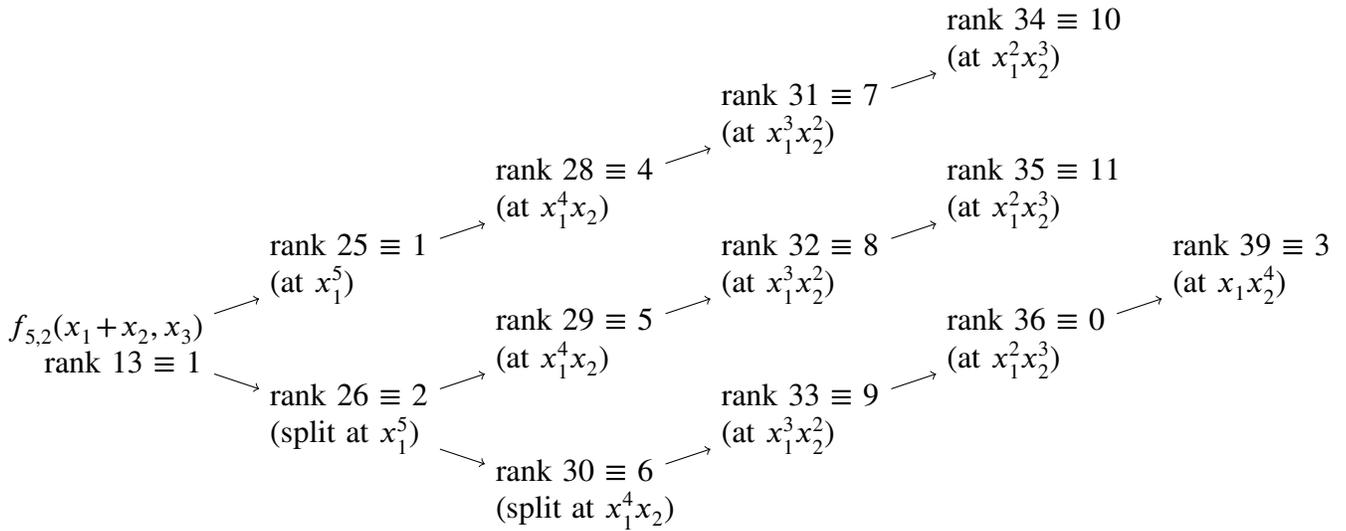
\begin{figure}[h]
\begin{tikzpicture}[grow=right]
 \node[ns] {\hspace{-18pt}  $f_{5,2}(x_1+x_2,x_3)$  rank~13~$\equiv$~1 }
    child[->] {
        node[ns] {rank~26~$\equiv$~2 (split~at~$x_1^5$)}        
            child {
                node[ns] {rank~30~$\equiv$~6 (split~at~$x_1^4 x_2$)}        
                child[edge from parent/.style={draw=none}] {
                }
                child {
                    node[ns] {rank~33~$\equiv$~9 (at~$x_1^3 x_2^2$)}        
                    child[edge from parent/.style={draw=none}] {
                    }
                        child {
                            node[ns] {rank~36~$\equiv$~0 (at~$x_1^2 x_2^3$)}        
                            child[edge from parent/.style={draw=none}] {
                            }
                            child {
                                node[ns ] {rank~39~$\equiv$~3 (at~$x_1 x_2^4$)}
                            }
                        }
                }
            }
            child {
                    node[ns] {rank~29~$\equiv$~5 (at~$x_1^4 x_2$)}        
                    child[edge from parent/.style={draw=none}] {
                    }
                    child {
                        node[ns] {rank~32~$\equiv$~8 (at~$x_1^3 x_2^2$)}        
                        child[edge from parent/.style={draw=none}] {
                        }
                        child {
                            node[ns] {rank~35~$\equiv$~11 (at~$x_1^2 x_2^3$)}
                        }
                    }
            }
    }
    child[->] {
        node[ns] {rank~25~$\equiv$~1 (at~$x_1^5$)}        
            child[edge from parent/.style={draw=none}] {
            }
            child {
                    node[ns] {rank~28~$\equiv$~4 (at~$x_1^4 x_2$)}        
                    child[edge from parent/.style={draw=none}] {
                    }
                    child {
                        node[ns] {rank~31~$\equiv$~7 (at~$x_1^3 x_2^2$)}        
                        child[edge from parent/.style={draw=none}] {
                        }
                        child {
                            node[ns] {rank~34~$\equiv$~10 (at~$x_1^2 x_2^3$)}
                        }
                    }
            }
    };
\end{tikzpicture}
\caption{Possible ranks for polynomials invariant under $\Gamma =\langle \omega I_2 \oplus \omega^2 I_1\rangle$ for $\omega$ a primitive fifth root of unity}
\end{figure}

\section{Future Directions} \label{s:future}

It is an open problem to determine precisely which $N$ are minimal embedding dimensions for nontrivial group-invariant, smooth, CR mappings from $S^{2n-1}$ to $S^{2N-1}$.  It is not even known whether D'Angelo's canonical mapping gives the smallest possible such $N$, i.e., whether there exists a non-constant, $\Gamma$-invariant, smooth, CR mapping from $S^{2n-1}$ to $S^{2N-1}$ with $N$ less than the minimal embedding dimension $N(\Gamma)$ of D'Angelo's map. Moreover, when $n=2$, our construction shows that the minimal embedding dimension can be any integer at least $2N({\Gamma})-1$. One naturally wonders if it is possible to achieve values between $N(\Gamma)$ and $2N({\Gamma})-1$.  

When $n>2$, Theorem 1 shows that for every $N \geq N(\Gamma)^2-2N(\Gamma)+2$, there exists a monomial, group-invariant CR mapping from $S^{2n-1}$ to $S^{2N-1}$ where $N$ is minimal. The examples and discussion in Section 4 suggest that a better bound is possible, and it is an interesting problem to determine all possible minimal embedding dimensions. 

Finally, we have focused on sphere mappings, but a natural extension is to generalize to hyperquadric mappings (see \cites{BEH1, BEH2,GLW, GLV} for work in this direction).

\end{document}